\documentclass[10pt]{amsart}
\usepackage{amssymb} 
\newtheorem{thm}{Theorem}[section]

\newtheorem{lem}[thm]{Lemma}
\newtheorem{prop}[thm]{Proposition}
\theoremstyle{definition}

\newtheorem{rem}[thm]{Remark}
\newtheorem{prob}[thm]{Problem}
\numberwithin{equation}{section}

\begin{document}
\title[minimal blocking sets and covering groups by subgroups]{Minimal blocking sets in $PG(n,2)$ and covering groups by subgroups}%
\author[A. Abdollahi, M. J. Ataei, A. Mohammadi Hassanabadi]{A. Abdollahi$^*$, M. J. Ataei, A. Mohammadi Hassanabadi\\
\sf {Department of Mathematics\\ University of Isfahan\\ Isfahan
81746-73441\\ Iran.}}
\thanks{This work was supported by the Excellence Center of
University of Isfahan for Mathematics. }
\subjclass{Mathematics Subject Classification:  51E21, 20D60}%
 \keywords{Blocking sets; projective spaces;  maximal irredundant
covers for groups; covering groups by subgroups}%
\thanks{$^*$Corresponding author: {\tt a.abdollahi@math.ui.ac.ir}}
\thanks{e-mail: \tt  mj.ataey@sci.ui.ac.ir{,}
aamohaha@yahoo.com}
\begin{abstract}
In this paper we prove that  a set of points $B$ of $PG(n,2)$ is a
minimal blocking set if and only if $\langle B\rangle=PG(d,2)$
with $d$ odd and $B$ is a set of $d+2$ points of $PG(d,2)$ no
$d+1$ of them in the same hyperplane. As a corollary to the latter
result we show that if $G$ is a finite  $2$-group and  $n$ is a
positive integer, then $G$ admits a $\mathfrak{C}_{n+1}$-cover if
and only if $n$ is even and $G\cong (C_2)^{n}$, where by a
$\mathfrak{C}_m$-cover for a group $H$ we mean  a set
$\mathcal{C}$ of size $m$ of maximal subgroups of $H$ whose
set-theoretic union is the whole $H$ and no proper subset of
$\mathcal{C}$ has the latter property and the intersection of the
maximal subgroups is core-free.  Also for all $n<10$ we find all
pairs $(m,p)$ ($m>0$ an integer and $p$ a prime number) for which
there is a blocking set $B$ of size $n$ in $PG(m,p)$ such that
$\langle B\rangle=PG(m,p)$.
\end{abstract} \maketitle
\section{\bf Introduction and results}
 Let  $G$ be a group. A set $\mathcal{C}$ of proper subgroups
of $G$ is called a cover for $G$ if its set-theoretic union is
equal to $G$. If the size of $\mathcal{C}$ is $n$, we call
$\mathcal{C}$ an $n$-cover for the group $G$. A cover
$\mathcal{C}$ for a group $G$  is called irredundant if no proper
subset of $\mathcal{C}$ is  a cover for $G$. A cover
$\mathcal{C}$ for a group $G$ is called core-free if the
intersection $D=\bigcap_{M\in\mathcal{C}}M$ of $\mathcal{C}$ is
core-free in $G$, i.e. $D_G=\bigcap_{g\in G} g^{-1}Dg$ is the
trivial subgroup of $G$. A cover $\mathcal{C}$ for a group $G$ is
called maximal if all the members of $\mathcal{C}$ are maximal
subgroups of $G$. A cover $\mathcal{C}$ for a group $G$ is called
a $\mathfrak{C}_n$-cover whenever $\mathcal{C}$ is an irredundant
maximal core-free $n$-cover for $G$ and in this case we say that
$G$ is a $\mathfrak{C}_n$-group.

Let $n$ be a positive integer. Denote by $PG(n,q)$  the
$n$-dimensional projective space over the finite field
$\mathbb{F}_q$ of order $q$. A blocking set  in $PG(n,q)$ is a set
of points that has nonempty intersection with every hyperplane of
$PG(n,q)$. A blocking set that contains a line is called trivial.
A blocking set  is called minimal if none of its proper subsets
are  blocking sets. For a blocking set $B$ in $PG(n,q)$ we denote
by $d(B)$ the least positive integer $d$ such that $B$ is
contained in a $d$-dimensional subspace of $PG(n,q)$. Thus
  $d(B)$ is equal to the (projective) dimension of the subspace spanned by $B$ in $PG(n,q)$.

For further studies in the topic of blocking sets see Chapter $13$
of the second edition of Hirschfeld's book \cite{HIR} and also see
\cite{Storme}.

The problem of covering a finite group with subgroups of a
specified order has been studied in \cite{paking} where also
bounds on the size of such covers was found. From Proposition 2.5
of \cite{paking} which is proved by  a deep theorem due to
Blokhuis \cite{blokhuis},  the following result easily follows.
We will require this as an auxiliary tool later.
\begin{thm}{\rm (See Proposition 2.5 of \cite{paking})} \label{paking} Let $p$ be a prime and let $G$ be a
finite $p$-group with a maximal irredundant $n$-cover. Then either
$n\geq \frac{3(p+1)}{2}$ or $n=p+1$.
\end{thm}
In section $2$ we give  relations between non-trivial minimal
blocking sets of size $n$ and $\mathfrak{C}_n$-groups. Also we
give a complete characterization of  minimal blocking sets in
$PG(n,2)$.

 In order to characterize
all $\mathfrak{C}_n$-groups, we first need  to know the structure
of an elementary abelian  $\mathfrak{C}_n$-group. This is
equivalent to find pairs $(m,p)$ $(m\in \mathbb{N}$ and  $p$ is a
prime number) for which $PG(m,p)$ contains a non-trivial  minimal
blocking set of size $n$ (See Propositions \ref{b1}, \ref{p3} and
\ref{p12}).

Nontrivial minimal blocking sets in $PG(2,p)$ of size
$\frac{3(p+1)}{2}$ exist for all odd primes $p$. Indeed, an
example is given by the projective triangle: the set consisting
of the points $(0,1,-s^2)$, $(1,-s^2,0)$, $(-s^2,0,1)$ with $s\in
\mathbb{F}_p$.

 In \cite{GS}, the
smallest non-trivial blocking sets of $PG(n,2)$ $(n\geq 3$) with
respect to $t$-spaces ($1\leq t\leq n-1$) are classified. A
complete classification of minimal blocking sets seems to be
impossible, however in this paper we determine all minimal
blocking sets  in $PG(n,2)$, namely we prove
\begin{thm}\label{thm2} A set of points $B$ of $PG(n,2)$ is a minimal blocking set if and only if
$\langle B\rangle=PG(d,2)$ with $d$ odd and $B$ is a set of $d+2$
points of $PG(d,2)$ no $d+1$ of them in the same hyperplane.
\end{thm}

For positive integers $n$ and $m$,  we denote the direct product
of $m$ copies of the cyclic group $C_n$ of order $n$  by
$(C_n)^m$.

Using Theorem \ref{thm2} we characterize all finite 2-groups
having a $\mathfrak{C}_{n+1}$-cover as follows:
\begin{thm}\label{application}
Let $G$ be a finite  $2$-group and let $n$ be a positive integer.
Then $G$ admits a $\mathfrak{C}_{n+1}$-cover if and only if $n$ is
even and $G\cong (C_2)^{n}$.
\end{thm}

 Groups with a $\mathfrak{C}_n$-cover for
$n=3,4,5$ and $6$ are characterized without appealing to the
theory of blocking sets; see \cite{scorza}, \cite{greco1},
\cite{BFS97} and \cite{AAJM}, respectively.

 In section $3$ we give some results
on $p$-groups ($p$ prime)  satisfying the property
$\mathfrak{C}_n$ for some positive integer $n$.

In section $4$ we characterize elementary abelian
$\mathfrak{C}_n$-groups for $n\in\{7,8,9\}$ as follows:
\begin{thm}\label{p-group7} Let $G$ be a $\mathfrak{C}_7$-group. Then $G$ is a $p$-group
for a prime number $p$  if and only if $G\cong (C_{2})^6$ or $
(C_{3})^4$.
\end{thm}
\begin{thm}\label{p-group8} Let $G$ be a $\mathfrak{C}_8$-group. Then $G$ is a $p$-group
for a prime number $p$  if and only if  $G\cong (C_{3})^4$ or
$(C_{7})^2$.
\end{thm}
\begin{thm}\label{p-group} Let $G$ be a $\mathfrak{C}_9$-group. Then $G$ is a $p$-group
for a prime number $p$  if and only if $G\cong (C_{2})^8$ or
$(C_3)^5$ or  $(C_{5})^3 $.
\end{thm}
In these characterizations we use Theorem  \ref{paking} as well as
some lemmas, the proof of which will be given  in section $3$.

We use Theorems \ref{p-group7}, \ref{p-group8} and \ref{p-group}
to give certain  non-trivial minimal blocking sets  in $PG(3,3)$
of sizes $7$ and 8; and in  $PG(4,3)$ of size $9$.
\begin{thm}\label{minimal}
 (a) Nontrivial minimal blocking sets of size $7$ exist in
 $PG(3,3)$.\\
             (b) Nontrivial minimal blocking sets of size $8$ exist in
             $PG(3,3)$.\\
             (c) Nontrivial minimal blocking sets of size $9$ exist in $PG(4,3)$.
\end{thm}
As a corollary to Theorems \ref{p-group7}, \ref{p-group8} and
\ref{p-group} and some known results we give in a table all pairs
$(m,p)$ ($m>0$ an integer and $p$ a prime number) for which there
is a blocking set $B$ of size $n<10$ in $PG(m,p)$ such that
$d(B)=m$.
\section{\bf Relations between  blocking sets and $\mathfrak{C}_n$-groups and characterization of minimal  blocking sets in $PG(n,2)$}

As we mentioned in section 1,   by a blocking set in $PG(n,q)$, we
mean a blocking set with respect to hyperplanes in $PG(n,q)$.

Now we give some notations and definitions as needed in the
sequel.  We denote the product of $n$ copies of $\mathbb{F}_q$
 by $(\mathbb{F}_q)^n$. We note that $(\mathbb{F}_q)^n$ is a vector
space of dimension $n$ over $\mathbb{F}_q$. If
$b=(b_1,\dots,b_n)\in (\mathbb{F}_q)^n$, we denote by $M_b$ the
set of elements $x=(x_1,\dots,x_n)\in (\mathbb{F}_q)^n$, such
that $b\cdot x=\sum_{i=1}^n b_ix_i$ is equal to zero. Note that
if $0\not=b$, then $M_b$ is an $(n-1)$-dimensional subspace of
the vector space $(\mathbb{F}_q)^n$ and every $(n-1)$-dimensional
subspace of $(\mathbb{F}_q)^n$ equals to $M_b$ for some non-zero
$b\in (\mathbb{F}_q)^n$. Since for every $0\not=\lambda\in
\mathbb{F}_q$, $M_b=M_{\lambda b}$, $M_{\mathfrak{p}}$ is
well-defined for every point $\mathfrak{p}$ of $PG(n-1,q)$, and
$M_{\mathfrak{p}}$ may be considered as a hyperplane in
$PG(n-1,q)$.  We now give some results which clarify the
relations between non-trivial minimal blocking sets of size $n$
and $\mathfrak{C}_n$-covers for groups.\\
The following Propositions \ref{b1}, \ref{p3} and \ref{p2} are
well-known and their proofs are straightforward.

\begin{prop}\label{b1}
Let $B$ be a set of points in $PG(n,q)$. Then $B$ is a blocking
set in $PG(n,q)$ if and only if the set $\mathcal{C}=\{M_b \;|\;
b\in B\}$ is a  $|B|$-cover for the abelian group
$(\mathbb{F}_q)^{n+1}$.
\end{prop}
\begin{prop}\label{p3}
Let $B$ be a set of points in $PG(n,q)$. Then $B$ is a minimal
blocking set in $PG(n,q)$ if and only if the set
$\mathcal{C}=\{M_b \;|\; b\in B\}$ is an irredundant $|B|$-cover
for the abelian group $(\mathbb{F}_q)^{n+1}$.
\end{prop}
\begin{rem}\label{rem1}
Note that if $q$ is prime, then  the cover $\mathcal{C}$ in the
statements of Propositions  \ref{b1} and \ref{p3} is a maximal
cover for $(\mathbb{F}_q)^{n+1}$.
\end{rem}
\begin{rem} It is easy to
see that a  (minimal) blocking set $B$ with $d(B)=d$ in $PG(n,q)$
can be obtained from a (minimal) blocking set in $PG(d,q)$.  So
if we adopt an induction process on $n$  to find all minimal
blocking sets $B$ in $PG(n,q)$, we must find only all those
 minimal blocking sets with $d(B)=n$.
\end{rem}
\begin{prop}\label{p2}
Let $B$ be a set of points in $PG(n,q)$. Then $B$ is a blocking
set with $d(B)=n$ if and only if the set $\mathcal{C}=\{M_b \;|\;
b\in B\}$ is a  core-free $|B|$-cover for the abelian group
$(\mathbb{F}_q)^{n+1}$.
\end{prop}
\begin{prop}\label{p12}
Let $p$ be a prime number and $n$ be a positive integer. Then a
finite $p$-group $G$ admits a $\mathfrak{C}_{n+1}$-cover if and
only if $G\cong (C_p)^{m+1}$ for some positive integer $m$ such
that $PG(m,p)$ has a minimal blocking set $B$ with $d(B)=m$ and
$|B|=n+1$.
\end{prop}
\begin{proof}
Let $G$ be a finite $p$-group admitting a
$\mathfrak{C}_{n+1}$-cover. Then $G$ has a maximal irredundant
core-free $(n+1)$-cover, $\mathcal{C}=\{M_i \;|\;
i=1,\dots,n+1\}$ say. Since the Frattini subgroup $\Phi(G)$ of
$G$ is contained in $M_i$ for every $i\in\{1,\dots,n+1\}$,
$\Phi(G)\leq D_G=1$, where $D$ is the intersection of the cover
$\mathcal{C}$. Hence $\Phi(G)=1$ and so $G$ is isomorphic to
$(C_p)^{m+1}$ for some positive integer $m$. Now Propositions
\ref{p3} and \ref{p2} and Remark \ref{rem1} complete the proof.
\end{proof}
Let $B$ be a set of points in $PG(n,q)$. Call any  $|B|\times
(n+1)$ matrix whose rows are generators  of points of $B$ a {\em
blocking matrix} of $B$ $($regard a point in
 $PG(n,q)$ as a one dimensional subspace in $(\mathbb{F}_q)^{n+1})$.

 Consider the following  properties of a  blocking  matrix $A$ of a set of points
  $B$ in $PG(n,q)$

  $(a)$ The $|B|\times 1$ column matrix $AX$ has at least one zero entry for every
$(n+1)\times 1$ column matrix $X$ with entries from
$\mathbb{F}_q$.

  $(b)$ For each $i\in \{1,\dots,|B|\}$, there is an $(n+1)\times 1$ matrix $X_i$ with
entries from $\mathbb{F}_q$ such that the $i$th entry of $AX_i$ is
zero and all the others are non-zero.

\begin{prop}\label{p4}
Let $B$ be a set of points in $PG(n,q)$ and let $A$ be any
blocking matrix of $B$. Then
 \begin{enumerate} \item $B$ is a  blocking set in $PG(n,q)$ with
$d(B)=rank(A)-1$ if and only if a blocking matrix of $B$ satisfies
the property (a).
\item $B$ is a minimal   blocking set in $PG(n,q)$ with $d(B)=rank(A)-1$ if and only
if a blocking matrix of $B$ satisfies the properties $(a)$ and
$(b)$.
\end{enumerate}
\end{prop}
\begin{proof}
It follows from Propositions \ref{p3} and \ref{p2}.
\end{proof}
In the following we apply a well-known  idea which is used in
coding theory to define an equivalence on  linear codes (see e.g,
pp. 50-51
in \cite{Hill}).\\

 Let $A_1$ and $A_2$ be two matrices of the same
size with entries from $\mathbb{F}_q$. We say that $A_1$ is {\em
equivalent} to $A_2$, if $A_2$ can be obtained from $A_1$ by a
sequence of operations of the following types:
\begin{itemize}
\item[(C1)] Permutation of the columns. \item[(C2)] Multiplication
of a column by a non-zero scalar from $\mathbb{F}_q$. \item[(C3)]
Addition of a scalar multiple of one column to another.
\item[(R1)] Permutation of the rows. \item[(R2)] Multiplication
of any row by a non-zero scalar.
\end{itemize}
\begin{thm}
Let $B$ be a minimal  blocking set in $PG(n,q)$ and let $A$ be a
blocking matrix of $B$. If $A'$ is a matrix obtained from $A$ by
one of the  operations (C1) to (R2), then (the points generated
by) the rows of $A'$ form a minimal blocking set $B'$ in $PG(n,q)$
with $d(B)=d(B')$ and $|B|=|B'|$.
\end{thm}
\begin{proof}
Using Proposition \ref{p4} and noting that
$$AX=x_1A_1+\dots +x_{n+1}A_{n+1},$$ where $A_1,\dots,A_{n+1}$ are
columns of $A$ and
  $X=\begin{bmatrix}
    x_1 \\ x_2 \\ \vdots\\  x_{n+1}
  \end{bmatrix}$, the proof is straightforward.
\end{proof}
We say that two minimal  blocking sets are {\em equivalent} if any
blocking matrix of one of them is equivalent to any blocking
matrix of the other.
\begin{thm}\label{thm1}
Let $A$ be any blocking matrix of a minimal  blocking set $B$ in
$PG(n,q)$,  let $k=|B|$ and $d=d(B)$.   Then  $A$ is equivalent to
a matrix of the form
  $\begin{bmatrix}
    I_{d+1} &|\;\;\;  \\
     ---   &| K \\
      L     &|\;\;\;
  \end{bmatrix},$
where $I_{d+1}$ is the $(d+1)\times (d+1)$ identity matrix,  $L$
is a $(k-d-1)\times (d+1)$ matrix and $K$ is a $k \times (n-d)$
matrix. Also $B$ can be obtained from a blocking set $\bar{B}$
with $d(\bar{B})=d$ in $PG(d,q)$ such that any blocking matrix of
$\bar{B}$ is equivalent to a matrix of the form $\begin{bmatrix}
    I_{d+1}   \\
     L
  \end{bmatrix},$ where $L$ is a $(k-d-1)\times (d+1)$ matrix.
\end{thm}
\begin{proof}
It is straightforward, see e.g., the proof of  Theorem 5.5 in p.
51 of \cite{Hill}.
\end{proof}

\noindent{\bf Proof of Theorem \ref{thm2}.} We prove the
following statement which is slightly more general than the statement of Theorem \ref{thm2}:\\
 A minimal blocking set $B$ with $d(B)=d$ in
$PG(n,2)$ exists if and only if $d$ is odd, $|B|=d+2$ and $B$ can
be obtained from a blocking set $\bar{B}$ with $d(\bar{B})=d$  in
$PG(d,2)$ such that any blocking matrix of $\bar{B}$ is
equivalent to a $(d+2)\times (d+1)$  matrix of the form
$$  \begin{bmatrix}
     &  & &I_{d+1} & & \\
        1&1&1&\cdots&1&1
  \end{bmatrix}.$$
 Let $B$ be a minimal
blocking set in $PG(n,2)$ with $d(B)=d$.
 By Theorem \ref{thm1}, $B$ can be obtained from a minimal blocking set $\bar{B}$  in
$PG(d,q)$ with $d(\bar{B})=d$ such that any blocking matrix of
$\bar{B}$ is equivalent to a matrix
 $A'=\begin{bmatrix}
    I_{d+1} \\
      L
  \end{bmatrix},$ where   $L$ is a $(k-d-1)\times (d+1)$ matrix, and $k=|B|=|\bar{B}|$.
   Let $
  \bold{a}=\begin{bmatrix}
    a_1 & a_2 &\cdots&a_{d+1} \\
      \end{bmatrix}$ be an arbitrary row of $L$. Note that the
      column   matrix $X_i$ which satisfies Property (2) for $A'$ in  Proposition
      \ref{p4} is unique, for every $i\in\{1,\dots,k\}$ and indeed $X_i$ is
      the $(d+1)\times 1$ matrix in which the $i$th entry is zero and all other entries
      are $1$.
   Thus $\bold{a}\cdot X_i\not=0$ and so
   $$a_1+\cdots+a_{i-1}+a_{i+1}+\cdots+a_{d+1}=1 \;\;\text{for all}\;\;
   i\in\{1,\dots,d+1\}. \eqno{(I)}$$
   On the other hand, by  Proposition \ref{p4}(1), the column matrix
 $A' \begin{bmatrix}
    1  \\
    1\\
    \vdots\\
    1
  \end{bmatrix}$ must have a zero entry. But all $(d+1)$ first
  entries are non-zero, so we have that $$a_1+\cdots+a_{d+1}=0. \eqno{(II)}$$
  Now it follows from $(I)$ and $(II)$  that $\bold{a}=\begin{bmatrix}
    1 & 1 &\cdots&1 \\
      \end{bmatrix}$. Therefore $L$ must have only one row which
      equals to $\begin{bmatrix}
    1 & 1 &\cdots&1 \\
      \end{bmatrix}.$ Now equality $(II)$ implies that
  $\underbrace{1+1+\cdots+1}_{d+1}=0$ from which it follows   that $d$ must be
  odd. This completes the proof. \hfill$\Box$\\

\noindent {\bf Proof of Theorem \ref{application}.} It follows
from  Proposition \ref{p12} and Theorem \ref{thm2}. \hfill $\Box$
\section{\bf  $p$-Groups with a $\mathfrak{C}_n$-cover}

We shall need the following lemma in the sequel.
\begin{lem}\label{lem2.2}{\rm (See  Lemma 2.2 of \cite{BFS97})}
Let $\Gamma=\{A_i \;:\; 1\leq i\leq m \}$ be an irredundant
covering of a group $G$ whose  intersection of the members  is
$D$.

 {\rm(a)} If $p$ is a prime, $x$ a $p$-element of $G$ and $\left| \left\{
i:x\in A_{i}\right\} \right|  =n$ , then either $x\in D$ or
$p\leq m-n$.

 {\rm (b)} $\underset{j\neq i}{\cap}A_{j}=D$  for all $\
i\in\left\{ 1,2,\dots,m\right\}$.

{\rm (c)} If  $\underset{i\in S}{\cap}A_{i}=D$  whenever
$\left|S\right|=n$, then $\left| \underset{i\in
T}{\cap}A_{i}:D\right| \leq m-n+1$ whenever $\left| T\right|
=n-1$ .

{\rm (d)} If $\Gamma$ is maximal and $U$ is an abelian minimal
normal subgroup of $G$. Then if $\left|\{i:U\subseteq
A_{i}\}\right|=n$, either $U\subseteq D$ or $\left|  U\right|
\leq m-n$.
\end{lem}
We now prove some key lemmas.
\begin{lem}\label{p} Let $G$ be a finite $p$-group having
a $\mathfrak{C}_n$-cover $\{M_i \;|\; i=1,\dots,n\}$. Then

{\rm (a)}  $p\leq n-1$.

{\rm (b)}  If $s$ is the integer such that    $1\leq s\leq n-2$
and $p=n-s$, then $\bigcap_{i\in S}M_i= 1$ for every subset $S$ of
$\{1,2,\dots,n\}$ with $|S|\geq s+1$.

{\rm (c)}  If $n=p+1$, then $G\cong (C_{p})^2$.
\end{lem}
\begin{proof}
Any blocking set $B$ of $PG(d,q)$ has at least $q+1$ points, and
equality holds if and only if $B$ is a line of $PG(d,q)$. This
corresponds, for $q=p$ prime, to points (a) and (c) of Lemma
\ref{p} (see also \cite[Proposition 2.5]{paking}). We give here a
group-theoretic proof for the points (a) and (c).

(a) Let $x$ be a $p$-element in $G$. Then by Lemma \ref{lem2.2}
(a), we have $p\leq n-m$, where $m=|\{i :  x \in M_{i}\}|$.
Therefore $p\leq n-1$.

 (b)  Let  $S\subseteq \{1,2,\dots,n\}$
with  $|S|=s+1$; and let  $N:= \bigcap_{i\in S}M_i$. Then $N
\unlhd G$, since $M_i \unlhd G$. Now suppose, for a contradiction,
that $N \neq 1$. Since $G=(\bigcup_{i\in
S}M_i)\bigcup(\bigcup_{j\notin S}M_j)$ and $|G:M_{k}|$ =$p$ for
every $k\in\{1,\dots,n\}$, by Lemma 3.2 of \cite{Tom2}, we have
$p\leq n-s-1$. This contradiction completes the proof of part (b).

 (c)  By Proposition \ref{p12},  $G$ is a finite  elementary
 abelian  $p$-group and by part (b),  $M_1 \cap M_2=1$. Thus $|G|=|G:M_1\cap M_2|=p^{2}$ and so $G\cong (C_{p})^2$.
 \end{proof}
 \begin{lem}\label{lem2.5}  Let $G=(C_{p})^{d}$ ($d\geq 2$ and
$p$ is a prime number) and suppose that $G$ has a
$\mathfrak{C}_n$-cover $\{M_i \;|\; i=1,\dots,n\}$. Let
  $T\subseteq \{1,2,\dots,n\}$.\\
{\rm (a)}   If $|T|=n-p$, then $|\bigcap_{i\in T}M_i|=1$ or $p$.\\
{\rm (b)} If $|T|=2$, then $|\bigcap_{i\in T}M_i|=p^{d-2}$.\\
{\rm (c)}  $\bigcap_{i\in T}M_i=1$ for some $T$ of size $d$.\\
{\rm (d)} If $\bigcap_{i\in S} M_i=1$ whenever $|S|=d$ then $p\leq
|\bigcap_{i\in T} M_i|\leq n-d+1$ whenever $|T|=d-1$.
\end{lem}
\begin{proof}  (a) By Lemma \ref{p}(b) $\bigcap_{i\in K}M_i=1$ for every subset $K$
of  $ \{1,2,\dots,n\}$ such that   $|K|= n-p+1$. Now  by Lemma
\ref{lem2.2} (c),
 $|\bigcap_{i\in T}M_i:\bigcap_{j\in K}M_j|\leq p$ and since $G$ is
 a $p$-group,  $|\bigcap_{i\in T}M_i|=1$ or $p$.

 (b)  Since each $M_j$ is a maximal subgroup of $G$, $|G:M_{j}|=p$.  Therefore $|G:\bigcap_{i\in T}M_i|=p^{2}$,
 and so $|\bigcap_{i\in T}M_i|=p^{d-2}$.

(c) Suppose that $M_i=M_{b_i}$, where $b_i \in PG(d-1,p)$ for
every $i\in\{1,\dots,n\}$. Then by Proposition \ref{p12},
$B=\{b_i \;|\; i=1,\dots,n\}$ is a minimal blocking set
  of size $n$ in $PG(d-1,p)$ such that $d(B)=d-1$. By Proposition
  \ref{p4}, $d=rank(A)$, where $A$ is any blocking matrix of $B$.
  Therefore there exists a subset $T\subseteq \{1,\dots,n\}$ such that $|T|=d$ and
   $\{b_i \;|\; i\in T\}$ is linearly independent. This implies that $\bigcap_{i\in T} M_i=1$, as required.

   (d) Since $|G|=p^d$, $|\bigcap_{i\in T} M_i|\geq p$ for all $T$
   with $|T|=d-1$. Now Lemma \ref{lem2.2}(c)   completes the proof.
\end{proof}
\section{\bf $p$-groups  having a $\mathfrak{C}_n$-cover for $n\in\{7,8,9\}$ }
 In this section we characterize all $p$-groups  having a
 $\mathfrak{C}_n$-cover  for $n=7,8$ and $9$. We denote by $[n]$
 the set $\{1,\dots,n\}$ and the set of all subsets of $[n]$ of
 size $m$ will be  denoted by $[n]^m$. We use the following
 results derived from the theory of blocking sets.
\begin{rem}\label{remi}
A minimal blocking set of $PG(2,q)$ has at most $q\sqrt{q}+1$
points \cite[Theorem 1 (i)]{BT}. From this it follows that  if
$(C_p)^3$ has a $\mathfrak{C}_n$-cover  then $n\leq p\sqrt{p}+1$.
\end{rem}
\begin{rem}\label{remii}
A minimal blocking set of $PG(3,p)$ with $p>3$ prime of size at
most $3(p + 1)/2 + 1$ is contained in a plane \cite[Theorem
1.4]{He}. This implies the non-existence of a
$\mathfrak{C}_9$-cover for $(C_5)^4$.
\end{rem}
\begin{rem}\label{remiii}
A minimal blocking set $B$ of $PG(3,q)$ has at most $q^2+1$ points
and equality holds if and only if $B$ is an ovoid \cite[Theorem 1
(ii)]{BT}. Also in \cite{MS} it has been proven that minimal
blocking sets of size $q^2$ in $PG(3,q)$ do not exist. Therefore
 there is no $\mathfrak{C}_9$-cover for $(C_3)^4$.
\end{rem}
 \begin{lem}\label{81} Let $G$ be a $3$-group. Then
$G$ is a $\mathfrak{C}_7$-group, if and only if  $G\cong (C_3)^4$.
\end{lem}
\begin{proof}
 Suppose that $G$ is a 3-group having a $\mathfrak{C}_7$-cover
 $\{M_i \;|\; i\in [7]\}$.
 By Proposition \ref{p12}, $G$ is an elementary abelian 3-group.
 By Lemma \ref{p}(b), $|G|\leq 3^5$.

 Since an elementary abelian group of order 9 has only four maximal subgroups, we
have $|G|\geq 27$.  From Remark \ref{remi}, it follows that
$\left|G\right|\neq 27$.

 Assume that $|G|=3^4$ so that $G\cong (C_3)^4$.
 Now  it is easy to check  (e.g. by {\sf GAP} \cite{GAP}) that  if $G=\left<a,b,c,d\right>$,
  then the set
  \begin{align*} \mathcal{C}=\{\left< a,b,c\right>,\left< a,c,d\right>,\left<b,c,d \right>,
   \left< a,b,d \right>,\left< a,b,c^{-1}d\right>,\left<a^{-1}b,c,d
  \right>,&\\ \left< ad,a^{-1}c,ab \right>\}
\end{align*}
   of maximal subgroups forms a $\mathfrak{C}_7$-cover for $G$.

 Now let $\left| G\right| =3^5$. Then  by Lemma \ref{lem2.5}(b),  $$\left|M_{i}\cap
M_{j}\right|=27 \;\; \text{for all distinct} \;\; i,j\in [7].
\eqno{(1)}$$ Since $|G|=3^5$, there is no subset $S\in [7]^3$ such
that $|\bigcap_{i\in S}M_i|\leq 3$. Thus $|\bigcap_{i\in S}M_i|\in
\{9,27\}$. Suppose, for a contradiction, that there exists  $L\in
[7]^3$ such that $|\bigcap_{i\in L}M_i|=27 \;\;\; (*).$ Let $L'\in
[7]^2$ such that $L' \cap L=\varnothing$. Thus by Lemma
\ref{p}(b), we have $|\bigcap_{i\in L\cup L'} M_i|=1$. Now if
$L''\subset L$ such that $|L''|=2$, then $(1)$ and $(*)$ imply
that $|\bigcap_{i\in L'' \cup L'} M_i|=1$. Since $|L'' \cup
L'|=4$, it follows that $|G|\leq 3^4$, which is a contradiction.
Therefore $|\bigcap_{i\in S}M_i|= 9$ for every $S\in [7]^3$ and
so we can apply  point (d) of  Lemma \ref{lem2.5} to  get that
$|\bigcap_{i\in T}M_i|=3$ for every $T\in [7]^4$. Now it follows
from Lemma \ref{p}(b) that $\bigcap_{i\in K} M_i=1$ for all
$K\subseteq [7]$ with $|K|\geq 5$. Now the inclusion-exclusion
principle implies that $|\bigcup^7_{i=1}M_{i}|=225$, which is
impossible. This completes the proof.
\end{proof}
\noindent{\bf Proof of Theorem \ref{p-group7}.} Let $G$ be a
$p$-group having a $\mathfrak{C}_7$-cover
 $\{M_i \;|\; i\in [7]\}$. By Proposition  \ref{p12},  $G$ is an elementary
abelian $p$-group. Now Theorem \ref{paking} implies that  $p=2$ or
$3$. If $p=2$, then it follows from Theorem \ref{application}
that $G\cong (C_2)^6$. If $p=3$, then Lemma \ref{81} implies
that  $G\cong (C_3)^4$, and the proof is  complete. \hfill $\Box$\\

\noindent{\bf Proof of Theorem \ref{minimal}(a).} Consider the
$\mathfrak{C}_7$-cover $\mathcal{C}$ for $(C_3)^4$ obtained  in
Lemma \ref{81} and   the set
\begin{align*}B=\{(1,0,0,0), (0,1,0,0), (0,0,1,0), (0,0,0,1),&\\
(1,1,0,0), (0,0,1,1),(1,-1,1,-1)\}
\end{align*} in $(F_3)^4$. It is easy to check
 (e.g. by {\sf GAP} \cite{GAP}) that  $\mathcal{C}=\{M_b \;|\; b\in B\}$.
 Now Propositions \ref{b1} and \ref{p3} imply that $B$ is a  minimal blocking set of size 7 in
$PG(3,3)$. \hfill $\Box$\\
\begin{lem}\label{C35}
The group $(C_3)^5$ has no $\mathfrak{C}_8$-cover.
\end{lem}
\begin{proof}
Suppose, for a contradiction, that $(C_3)^5$ has a
$\mathfrak{C}_8$-cover $\{K_1,\dots,K_8\}$, where $K_i=M_{b_i}$,
$b_i\in (\mathbb{F}_3)^5$. Then by Propositions \ref{p3} and
\ref{p2}, $B=\{b_i \;|\; i=1,\dots,8\}$ is a minimal blocking set
for $PG(4,3)$ with $d(B)=4$. Now it follows from Theorem
\ref{thm1} that a blocking matrix of $B$ is equivalent to a
matrix as follows $$A=\begin{bmatrix}
  {\mathbf e}_1 & {\mathbf e}_2 & {\mathbf e}_3 & {\mathbf e}_4 & {\mathbf e}_5 & {\mathbf x}_1 & {\mathbf x}_2  &
  {\mathbf x}_3
\end{bmatrix}^T,$$
where ${\mathbf e}_i$ is the vector in $({\mathbb F}_3)^5$ whose
$i$-th entry is $1$ and the others are zero. Since the matrix
$$A\begin{bmatrix}
  1 & 1 & 1 & 1 & 1
\end{bmatrix}^T$$
must have at least one zero entry, $${\mathbf x}_i\begin{bmatrix}
  1 & 1 & 1 & 1 & 1
\end{bmatrix}^T=[0], \eqno{(*)}$$
for some $i\in\{1,2,3\}$. By permuting the rows ${\mathbf
x}_1,{\mathbf x}_2,{\mathbf x}_3$, if necessary, we may assume
that $i=1$. Now $(*)$ implies that the sum of entries of
${\mathbf x}_1$ is zero. It follows that there are, up to column
permutations, only the following 4 vectors in $({\mathbb F}_3)^5$
with the latter property:
$$[-1,1,1,1,1],[0,1,-1,1,-1],[0,0,1,1,1],[0,0,0,1,-1].\eqno{(\#)}$$
These column permutations will not change the {\em set} of the top
5 rows
 of the blocking matrix $A$, as they are the rows of the $5\times 5$ identity
 matrix. Therefore $PG(4,3)$ has a minimal blocking set
 $$\{{\mathbf e}_1, {\mathbf e}_2, {\mathbf e}_3, {\mathbf e}_4, {\mathbf e}_5 ,\mathbf{x},\mathbf{y},\mathbf{z}\},$$ where
$\mathbf{x}$ is one of the vectors in $(\#)$ and ${\mathbf
y},\mathbf{z}\in(\mathbb{F}_3)^5$. Now Propositions \ref{p3} and
\ref{p2} imply that the maximal subgroups
$$M_1=M_{\mathbf{e}_1},M_2=M_{\mathbf{e}_2},M_3=M_{\mathbf{e}_3},M_4=M_{\mathbf{e}_4},
M_5=M_{\mathbf{e}_5},M_6=M_{\mathbf{x}},M_7=M_{\mathbf{y}},M_8=M_{\mathbf{z}}
\eqno{(**)}
$$ forms a $\mathfrak{C}_8$-cover of $(C_3)^5$. It is not hard to show (e.g., by {\sf GAP}
\cite{GAP}) that for every choice of the vector $\mathbf{x}$ from
$(\#)$ and for all non-zero vectors $\mathbf{y},\mathbf{z}\in
(\mathbb{F}_3)^5$, $(**)$ is not an irredundant cover, a
contradiction. This completes the proof.

We give here the proof of the latter claim when
$\mathbf{x}=[-1,1,1,1,1]$. Firstly we determine the vectors
${\mathbf e}_1,\dots, {\mathbf e}_5$ and the group $(C_3)^5$ in
{\sf GAP} as the following permutations and the group:
\begin{verbatim}
a:=(1,2,3);;b:=(4,5,6);;c:=(7,8,9);;d:=(10,11,12);;
e:=(13,14,15);; C35:=Group(a,b,c,d,e);;
\end{verbatim}
This means that we have considered the permutations
\verb+a,b,c,d,e+ instead of the vectors ${\mathbf e}_1,\dots,
{\mathbf e}_5$, respectively and so, for example, the vector
$\mathbf{x}$ is corresponded to the permutation
\verb+a^-1*b*c*d*e+ (written in {\sf GAP} command form). Now the
maximal subgroups $M_1,\dots,M_6$ are as follows in {\sf GAP}:
\begin{verbatim}
M1:=Group(b,c,d,e);;M2:=Group(a,c,d,e);;
M3:=Group(a,b,d,e);;M4:=Group(a,b,c,e);;M5:=Group(a,b,c,d);;
M6:=Group(a*b,b*c^-1,c*d^-1,d*e^-1);;
\end{verbatim}
We now produce all unordered pairs $\{M_7,M_8\}$ of maximal
subgroups
 of $(C_3)^5$ such that $$\{M_1,\dots, M_6, M_7,
M_8\}$$ is an $8$-cover of $(C_3)^5$.
\begin{verbatim}
T:=MaximalSubgroups(C35);; D:=Difference(T,[M1,M2,M3,M4,M5,M]);;
C:=Combinations(D,2);; K:=Union(M1,M2,M3,M4,M5,M);;
F:=Filtered(C,i->Size(Union(K,Union(i)))=3^5);;
B:=List(F,i->Union([M1,M2,M3,M4,M5,M],i));;
\end{verbatim}
Therefore the set \verb+B+ contains all $8$-covers of $(C_3)^8$
which contain $M_1,\dots,M_6$. It remains to check whether there
is  an {\em irredundant} cover of $B$ or not. The following
program collect all irredundant  covers from \verb+B+ into the
set \verb+R+ (if there is any).
\begin{verbatim}
R:=[ ]; for i in [1..Size(B)] do Q:=Combinations(B[i],7); if (3^5
in List(Q,i->Size(Union(i))))=false then Add(R,B[i]); fi; od;
\end{verbatim}
Finally we see that the set \verb+R+ is empty for this choice of
the vector $\mathbf{x}$. Similarly, for the other selections, we
get that \verb+R+ is empty. This proves the claim.
\end{proof}
\begin{lem}\label{cover81} Let $G$ be a $3$-group. Then
$G$ is a $\mathfrak{C}_8$-group, if and only if  $G\cong (C_3)^4$.
\end{lem}
\begin{proof}  Suppose that $G$ is a 3-group having a
$\mathfrak{C}_8$-cover $\{M_i \;|\; i\in [8]\}$.
 Proposition  \ref{p12} implies that $G$ is an elementary abelian 3-group.
By Lemma \ref{lem2.5}(b)  $$|G:M_{i}\cap M_{j}|=9 \; \text{for
all distinct}\;  i,j\in [8] \eqno{(I)}$$ and Lemma \ref{p}(b)
implies that $$\text{for every subset}\; T\subseteq [8] \;
\text{with} \; |T|\geq 6, \; \bigcap_{i\in T}M_i=1. \eqno{(II)}$$
It follows that $\left|G\right|\leq 3^6$. Since an elementary
abelian group of order 9 has only four maximal subgroups, we have
$\left| G\right|\geq 27$ and it follows from Remark \ref{remi}
that $\left|G\right|\neq 27$.

Assume that $|G|=3^4$ so that $G\cong (C_3)^4$.
 Now  it is easy to check  (e.g. by {\sf GAP} \cite{GAP}) that  if $G=\left<a,b,c,d\right>$,
  then the set
  \begin{align*} \mathcal{D}=\{\left< a,b,c\right>,  \left< a,c,d\right>,  \left<b,c,d \right>,
   \left< a,b,d \right>,   \left< a,b,c^{-1}d\right>, \left<a^{-1}b,c,d
  \right>,&\\  \left< d,ac,b \right>, \left< ad,c,a^2b
  \right>\}
\end{align*}
    of maximal subgroups forms a $\mathfrak{C}_8$-cover for $G$.

It follows from Lemma \ref{C35} that $|G|\not=3^5$.

 Now let $|G|=3^6$. Then   $(I)$ implies that
 for every $K\in [8]^3$ we have  $|\bigcap_{i\in K}M_i|=27$ or
$81$.

We now prove that $$|\bigcap_{i\in K}M_i|=27 \;\; \text{for all}
\; K \in [8]^3.\eqno{(III)}$$ Suppose, for a contradiction, that
there exists  $L\in [8]^3$ such that
$$|\bigcap_{i\in L}M_i|=81.\eqno{(*)}$$ Let $L'\in [8]^3$ such
that $L' \cap L=\varnothing$. Thus by $(II)$, we have
$|\bigcap_{i\in L\cup L'} M_i|=1$. Now if $L''\subset L$ such that
$|L''|=2$, then $(I)$ and $(*)$ imply that $|\bigcap_{i\in L''
\cup L'} M_i|=1$. Since $|L'' \cup L'|=5$, it follows that
$|G|\leq 3^5$, which is a contradiction.

Now $(III)$ yields that for   every   $W\in [8]^4$, we have
$|\bigcap_{i\in W}M_i|=9$ or $27$. By a similar argument as in the
latter paragraph, one can prove that
$$|\bigcap_{i\in W}M_i|=9  \;\; \text{for all} \;\; W \in [8]^4. \eqno{(IV)}$$

Since $G\cong(C_3)^6$ and $(II)$ holds, we can apply Lemma
\ref{lem2.5}(d) for the cover and so
$$|\bigcap_{i\in T}M_i|=3 \;\text{for all}\; T\in [8]^5.
\eqno{(V)}$$ Thus since $G=\bigcup_{i=1}^8 M_i$, it follows from
the inclusion-exclusion principle and the relations $(I)$,
$(II)$, $(III)$, $(IV)$, $(V)$ that $|G|=705$, which is
impossible. This completes the result.
\end{proof}

\noindent{\bf Proof of Theorem \ref{p-group8}.} Let $G$ be a
$p$-group having a $\mathfrak{C}_8$-cover $\{M_i \;|\; i\in
[8]\}$. By Proposition \ref{p12}, $G$ is an elementary abelian
$p$-group. Now Theorems \ref{paking} and \ref{application} imply
that $p=3$ or $p=7$. If $p=3$, then by Lemma \ref{cover81}, we
conclude that $G\cong (C_3)^4$. If $p=7$, then Lemma  \ref{p}(c)
yields that $G\cong (C_{7})^2$.

For the converse if $G\cong (C_{7})^2$, then $G$ is a
$\mathfrak{C}_8$-group, since it has exactly 8 maximal subgroups.
If $G\cong (C_3)^4$, then  Lemma \ref{cover81} completes the
proof. \hfill $\Box$\\

\noindent{\bf Proof of Theorem \ref{minimal}(b).} Consider the
$\mathfrak{C}_8$-cover $\mathcal{D}$ for $(C_3)^4$ obtained  in
Lemma \ref{cover81} and   the set \begin{align*} B=\{(1,0,0,0),
(0,1,0,0), (0,0,1,0), (0,0,0,1),&\\(1,1,0,0), (0,0,1,1),
(2,0,1,0),(1,1,0,2)\}.
\end{align*}
It is easy to check
 (e.g. by {\sf GAP} \cite{GAP}) that  $\mathcal{D}=\{M_b \;|\; b\in B\}$.
 Now Propositions \ref{b1} and \ref{p3} imply that $B$ is a  minimal blocking set of size 8 in
$PG(3,3)$. \hfill $\Box$\\

\begin{rem}
The examples of blocking sets in Theorem \ref{minimal} (a) and (b)
are particular examples of the following general families.
Example (a) of Theorem \ref{minimal} belongs to a family of
minimal blocking sets of $PG(3,q)$ of size $2q+1$ constructed in
\cite[Examples 1]{Ta} (see also \cite[p. 171 Examples (d)]{He}).
Also, any example of this family, when $q=p$ is prime, provides a
$\frak{C}_{2p+1}$-cover of $(C_p)^4$. In the same paper other
families of blocking sets of $PG(3,q)$ are presented that can be
used to obtain other examples of covers of $(C_p)^4$. Whereas in
\cite[Corollary 3.6 (b)]{He} are described examples of minimal
blocking sets generating the whole space in $PG(d,p)$ for any $p$
prime and any $d\geq 3$.  Example (b) of Theorem \ref{minimal}
seems to belong to a family of minimal blocking sets of $PG(3,q)$
($q$ odd) of size $2q+2$ constructed in \cite[Theorem 3.1 and
Remarks (c)]{He}. Also, any example of this family, when $q=p$ is
an odd prime, produces a $\frak{C}_{2p+2}$-cover of $(C_p)^4$.
\end{rem}

\begin{lem}\label{5-group} Let $G$ be a $5$-group. Then
$G$ is a $\mathfrak{C}_9$-group if and only if  $G\cong (C_5)^3$.
\end{lem}
\begin{proof}  Suppose that $G$ is a 5-group.
 By Proposition \ref{p12}, $G$ is an elementary abelian 5-group.
By Lemma \ref{lem2.5}(b) $$|G:M_{i}\cap M_{j}|=25 \; \text{for
all distinct} \; i,j\in [9]. \eqno{(1)}$$  Now Lemma \ref{p}(b)
implies that  $$\text{for every } \;  T\subseteq [9] \; \text{such
that} \;  |T|\geq 5, \; |\bigcap_{i\in T}M_i|=1. \eqno{(2)}$$
Therefore $\left|G\right|\leq 5^5$. Also by Lemma \ref{lem2.5}(a)
$$|\bigcap_{i\in
W}M_i|\in\{1,5\} \;\;\text{for all}\;   W\in [9]^4. \eqno{(3)}$$
 Since an elementary abelian group of order 25 has only six
maximal subgroups, we have $|G|\geq 5^3$.

Assume that $|G|=5^3$ so that $G\cong (C_5)^3$. As the projective
triangle in $PG(2,5)$ is a minimal blocking set of size 9,
Proposition \ref{p12} implies that $(C_5)^3$ is a
$\mathfrak{C}_9$-group. In fact
   if $G=\left<a,b,c\right>$,
  then the   set
  \begin{align*}\big\{\left<a,b\right>,  \left< a,c\right>,  \left<b,c \right>,
   \left< b^3c,a \right>,   \left< a^3c,a^4b\right>,&\\  \left<a^2c,ab
  \right>,  \left< b^2c,a \right>, \left< a^3c,ab
  \right>,  \left< a^2c,a^4b\right>\big\}. \end{align*}
of maximal subgroups    forms a $\frak{C}_9$-cover for $G$.

 It follows from Remark \ref{remii} that $|G|\not=5^4$.

 Now if $|G|=5^5$, then since (2) holds,  Lemma \ref{lem2.5}(d) implies  that $$\text{for every
subset} \; W\in [9]^4, \;\; |\bigcap_{i\in W}M_i|=5. \eqno{(4)}$$
Since $|G|=5^5$, it follows from  $(1)$ that
$$|\bigcap_{i\in K} M_i|=5^3 \;\;\text{for all} \;\;  K\in [9]^2, \eqno{(5)}$$ and so
   $|\bigcap_{i\in K}M_i|\in \{25,125\}$ for every  $K \in [9]^3$.

We prove  that $$|\bigcap_{i\in K}M_i|=25  \;\;\text{for all} \;
 K\in [9]^3. \eqno{(6)}$$ Since otherwise there exists $L\in
[9]^3$ such that $|\bigcap_{i\in L}M_i|=125$. Let $L'\in [9]^2$
such that $L' \cap L=\varnothing$. Then $(5)$ and $(2)$ imply that
$$\bigcap_{i\in L''\cup L'}M_i=\bigcap_{i\in L\cup L'} M_i=1$$ for every
$L''\subset L$ of size $2$. This implies that $|G|\leq 5^4$, which
is a contradiction.

Now using $(1)$, $(2)$, $(4)$, $(5)$ and $(6)$,  it follows from
the inclusion-exclusion principle that $|G|=2665$, which is a
contradiction.
\end{proof}
\begin{lem}\label{C36}
The group $(C_3)^6$ is not a $\mathfrak{C}_9$-group.
\end{lem}
\begin{proof}
By a similar argument as in Lemma \ref{C35}, one may prove the
lemma. Note that here  the vectors in $(\mathbb{F}_3)^6$ whose sum
of its entries is zero, up to column permutations, are the
following:
$$
[1,1,1,1,1,1],[-1,1,-1,1,-1,1],[0,-1,1,1,1,1],[0,0,1,-1,1,-1],
[0,0,0,1,1,1],[0,0,0,0,1,-1].\eqno{(*)} $$
  Therefore by a similar argument as in Lemma
\ref{C35}, we complete the proof by proving that there is no
minimal blocking set of size 9 for $PG(5,3)$ containing the
standard basis vectors $\mathbf{e}_1,\dots,\mathbf{e}_6$ and one
of the above vectors in $(*)$.
\end{proof}
\begin{lem}\label{3-group} Let $G$ be a $3$-group. Then
$G$ is a $\mathfrak{C}_9$-group, if and only if  $G\cong  (C_3)^5
$.
\end{lem}
\begin{proof} Suppose that $G$ is a 3-group having a
$\mathfrak{C}_9$-cover.
 Therefore  by Lemma \ref{lem2.5} $$|G:M_i|=3 \; \text{and}\; |G:M_i\cap M_j|=9 \; \text{for
all distinct} \; i,j\in [9]. \eqno{(i)}$$
 By Proposition \ref{p12}, $G$ is an elementary abelian  3-group. By Lemma \ref{p}(b) we have
 $$\text{for all} \; T \subseteq [9] \; \text{such that} \; |T|\geq
 7, \; |\bigcap_{i\in T}M_i|=1. \eqno{(ii)}$$
 Therefore  $|G |=3^d$, where $d\leq 7$.
  Since  an elementary abelian group of order $9$ has only four
maximal  subgroups, $|G|\geq 27$ and it follows from Remarks
\ref{remi} and \ref{remiii} that $|G|\not\in\{27,81\}$.  Therefore
$5\leq d\leq  7$.

 Assume that $|G|=3^5$ so that  $G\cong (C_3)^5$.
Now  it is easy to see (e.g. by {\sf GAP} \cite{GAP} )  that  if
$G=\left<a,b,c,d,e\right>$,
  then the  set
  \begin{align*}
 \mathcal{F}=\big\{ \left< a,b,c,d\right>,  \left< a,b,c,e\right>,  \left<a,b,d,e \right>, \left<a,c,d,e\right>,
 \left<b,c,d,e\right>,&\\
    \left<a^{-1}c,b,d,e\right>,  \left< b^{-1}c,a,d,e\right>,
      \left< de,c,b,a \right>,\left<a^{-1}e,a^{-1}d,c,ab \right>\big\}
      \end{align*}
    of maximal subgroups forms a $\mathfrak{C}_9$-cover for $G$.

Now assume that  $|G|=3^7$. Then it follows from $(i)$ that  for
every  $K \in [9]^3$,  we have $|\bigcap_{i\in K}M_i|= 3^4$ or
$3^5$. We now  prove  that $$|\bigcap_{i\in K}M_i|=3^4
\;\;\text{for all} \;\; K\in [9]^3. \eqno{(*)}$$ Suppose, for a
contradiction, that there exists $L\in [9]^3$ such that
$|\bigcap_{i\in L}M_i|=3^5$. Let $L'\in [9]^4$ such that $L' \cap
L=\varnothing$. Then it follows from $(i)$ and $(ii)$ that
$|\bigcap_{i\in L'' \cup L'}M_i|=1$ for every $L'' \subset L$ of
size $2$. This implies that $|G|\leq
3^6$, a contradiction. \\

Now since $|G|=3^7$, it follows from $(*)$ that  for every  $V\in
[9]^4$,  we have $|\bigcap_{i\in V}M_i|= 3^4$ or $3^3$. By a
similar argument as in the previous paragraph one can prove that
for all  $V\in [9]^4$, all $W\in [9]^5$ and all $T\in [9]^6$
  $$|\bigcap_{i\in V}M_i|=27, \; \;\;\; \; |\bigcap_{i\in
  V}M_i|=9 \;\;\; \text{and} \;\;\; |\bigcap_{i\in T}M_i|=3
. \eqno{(**)}$$  Now using $(i)-(ii)$ and $(*)-(**)$, it follows
from the inclusion-exclusion principle that
   $|G|=2125$, which is a contradiction.
\end{proof}

\noindent {\bf Proof of Theorem \ref{p-group}.} Let $G$ be a
$p$-group having a $\mathfrak{C}_9$-cover $\{M_i \;|\;
i=1,\dots,9\}$. By Proposition \ref{p12},  $G$ is an elementary
abelian $p$-group. By Theorem \ref{paking}, $p\leq 5$. Now  Lemmas
\ref{5-group},  \ref{3-group}  and Theorem \ref{application}
complete the proof. \hfill $\Box$\\

\noindent{\bf Proof of Theorem \ref{minimal}(c).} Consider the
$\mathfrak{C}_9$-cover $\mathcal{F}$ for $(C_3)^5$ obtained  in
Lemma \ref{3-group} and  the set
\begin{align*} B=\{(1,0,0,0,0), (0,1,0,0,0), (0,0,1,0,0), (0,0,0,1,0),&\\
(0,0,0,0,1),(1,0,1,0,0), (0,1,1,0,0), (0,0,0,1,-1),
(0,1,0,-1,-1)\}.
\end{align*}
\\ It is easy to check
 (e.g. by {\sf GAP} \cite{GAP}) that  $\mathcal{D}=\{M_b \;|\; b\in B\}$.
 Now Propositions \ref{b1} and \ref{p3} imply that $B$ is a  minimal blocking set of size 9 in
$PG(4,3)$. \hfill $\Box$\\

We end this paper by posing
\begin{prob}\label{prob}
Let $n$ be a positive integer. Find all pairs $(m,p)$ ($m>0$ an
integer and $p$ a prime number) for which there is a blocking set
$B$ of size $n$ in $PG(m,p)$ such that $d(B)=m$.
\end{prob}
The following table contains the complete answers for $n<10$ to
Problem \ref{prob}  which have been exerted  from the results of
the present paper together with some known ones:
\hspace{-13mm}$$\begin{tabular}{|c|c|c|c|c|c|c|c|}\hline
 $n$ & 3 & 4 & 5 & 6 & 7 & 8 &9 \\ \hline
  $(m,p)$ & (1,2)& (1,3)& (3,2)& (1,5),(2,3)& (5,2), (3,3)& (1,7), (3,3)& (7,2),(2,5),(4,3) \\ \hline
  Reference & \cite{scorza} & \cite{greco1} & \cite{BFS97} & \cite{AAJM} & Theorem \ref{p-group7} & Theorem \ref{p-group8} & Theorem \ref{p-group} \\ \hline
\end{tabular}$$
\noindent{\bf Acknowledgments.} The authors are grateful to L.
Storme for his  valuable suggestions on some parts of this paper.
They are indebted to the referee for pointing out a serious
mistake in the original version of the paper as well as quoting
some useful results from the theory of blocking sets by which the
use of {\sf GAP} has been reduced.

\end{document}